\documentclass[12pt,a4paper]{amsart}
\usepackage{amsmath}
\usepackage{amssymb}
\usepackage{amsthm}
\usepackage{amsfonts}
\usepackage{mathrsfs}
\usepackage{color}
\usepackage[colorlinks,linkcolor=red,anchorcolor=blue,citecolor=green]{hyperref}
\usepackage{cleveref}
\usepackage[margin=1.0in]{geometry}
\usepackage{esint}

\numberwithin{equation}{section}

\newtheorem{theorem}{Theorem}[section]
\newtheorem{lemma}[theorem]{Lemma}

\newtheorem{proposition}[theorem]{Proposition}

\theoremstyle{definition}
\newtheorem{definition}[theorem]{Definition}

\theoremstyle{remark}
\newtheorem{remark}[theorem]{Remark}

\theoremstyle{proof}
\newtheorem*{pfthm}{Proof of Theorem \ref{thm-weakregular}}
\newtheorem*{pfthm2}{Proof of Theorem \ref{thm-riemcone}}

\newcommand{\dd}{\partial}

\newcommand{\leqs}{\leqslant}
\newcommand{\geqs}{\geqslant}

\newcommand{\rean}{\mathbb{R}}
\newcommand{\iv}{^{-1}}

\newcommand{\sphere}{\mathbb{S}}

\DeclareMathOperator{\Ric}{Ric}

\DeclareMathOperator{\Vol}{Vol}

\title{New exotic examples of Ricci limit spaces}

\author{Xilun Li}
\address{School of Mathematical Sciences\\
Peking University\\
Beijing\\ 100871\\ China}
\email{lxl28@stu.pku.edu.cn}

\author{Shengxuan Zhou}
\address{Institut de Math\'{e}matiques de Toulouse\\
	 UMR 5219, Universit\'{e} de Toulouse\\
	CNRS, UPS, 118 Route de Narbonne, F-31062 Toulouse Cedex 9, France}
\email{zhoushx98@outlook.com}

\thanks{}
\date{}
\dedicatory{}

\begin{document}

\maketitle

\begin{abstract}
 For any integers $m\geqs n\geqs 3$, we construct a Ricci limit space $X_{m,n}$ such that for a fixed point, some tangent cones are $\rean^m$ and some are $\rean^n$. This is an improvement of Menguy's example\cite{MR1829644}. Moreover, we show that for any finite collection of closed Riemannian manifolds $(M_i^{n_i},g_i)$ with $\Ric_{g_i}\geqs(n_i-1)\geqs1$, there exists a collapsed Ricci limit space $(X,d,x)$ such that each Riemannian cone $C(M_i,g_i)$ is a tangent cone of $X$ at $x$. 
\end{abstract}

\maketitle
\tableofcontents

\section{Introduction}
\label{introduction}

Consider the measured Gromov-Hausdorff limit spaces as following:
\begin{align*}
    (M^n_{i} ,g_i ,\nu_i ,p_i ) \stackrel{GH}{\longrightarrow} (X^k ,d ,\nu ,p ) ,\;\; \Ric_{g_i} \geqs -\lambda ,\;\; \nu_i = \frac{1}{\Vol (B_1 (p_i)) } d\Vol_{g_i },
\end{align*}
where $k\in\mathbb{N}$ is the rectifiable dimension of $(X ,d ,\nu )$, which is the unique integer $k$ such that the limit is $k$-rectifiable. The existence of such a $k$ is proved by Colding-Naber\cite{MR2950772}. Moreover, the strong regular set $\mathcal{R}_k(X)$ is a $\nu$-full measure set. Actually, there are two versions of regular sets on $X$\cite{MR1484888}. For $l=1,\cdots ,n$, the weak regular set of $(X,d)$ can be defined by
\begin{align*}
  \mathcal{WR}_l (X) =  \left\{ x\in X : \textrm{ there exists a tangent cone at $x$ isometric to $\mathbb{R}^l$ } \right\},
\end{align*}
and the strong regular set of $(X,d)$ can be defined by
\begin{align*}
  \mathcal{R}_l (X) =  \left\{ x\in X : \textrm{ every tangent cone at $x$ isometric to $\mathbb{R}^l$ } \right\}.
\end{align*}

Cheeger-Colding\cite{MR1484888} shows that in the noncollapsing case, i.e. $\Vol(B_1(p_i))>v>0$ uniformly, two versions coincide. Moreover, the rectifiable dimension and the Hausdorff dimension of the limit space are both equal to $n$. However, in collapsing case, i.e. $\Vol(B_1(p_i))\to0$, many things are quite different. Pan-Wei\cite{MR4431126} shows that the Hausdorff dimension may be larger than the rectifiable dimension, and the Hausdorff dimension can be non-integers. Menguy\cite{MR1829644} shows that  the weak regular set may be not equal to the strong regular set. However, it was still not known whether the intersection of weakly regular sets of different dimensions can be non-empty. In this paper, we construct the first example that shows that the intersection of weakly regular sets of different dimensions can be non-empty. This is an improvement of Menguy's example \cite{MR1829644}.

\begin{theorem}\label{thm-weakregular}
Let $m\geqs n\geqs 2$ be integers. Then there exists a sequence of $(m+n+3)$-dimensional complete Riemannian manifolds $(M_i ,g_i,p_i)$ with $\Ric_{g_i} \geqs 0 $ converging to $(X_{m+1,n+1},d,x)$, such that 
$$ \mathcal{WR}_{m+1} (X) \cap \mathcal{WR}_{n+1} (X) \neq \emptyset .$$
\end{theorem}
\begin{remark}
    For this example, the rectifiable dimension of the limit space is $m+n+1$.
\end{remark}

Actually, using the same technique we can also prove the following stronger statement.\footnote{It was suggested by one of the referees.} It shows the possible non-uniqueness of tangent cones at a single point in the collapsed setting when the tangent cones are also Riemannian cones.
\begin{theorem}\label{thm-riemcone}
    Given a finite collection $(M_i^{n_i},g_i)$ $(i=1,\cdots,N)$ of closed Riemannian manifolds satisfying $\Ric_{g_i}\geqs(n_i-1)\geqs1$ for $1\leqs i\leqs N$. Then there exists a sequence of $(\sum_i n_i+5)$-dimensional complete Riemannian manifolds with nonnegative Ricci curvature converging to $(X,d,x)$, such that each Riemannian cone $C(M^{n_i},g_i)$ is isometric to a tangent cone of $X$ at $x$.
\end{theorem}


\textbf{Acknowledgements} We want to express our sincere gratitude to our advisor, Professor Gang Tian, for his
helpful suggestions, patient guidance and revision of the earlier verison. We thank Yanan Ye for valuable comments and discussions on the earlier draft of this paper. We are also grateful to Wenshuai Jiang for suggesting this problem. We would also like to thank the referees for careful reading and valuable suggestions, which help us to improve the presentation of this paper and strengthen our results. Authors are supported by National Key R\&D Program of China 2020YFA0712800.

\section{Triple warped products}
\label{sectionlocalmodel}
 
In this section, we recall the Ricci curvature of triple warped products.
 
 Let $\varphi ,\phi , \rho $ be smooth nonnegative functions on $[0,\infty )$ such that $\varphi ,\phi , \rho $ are positive on $(0,\infty )$,
 \begin{eqnarray} \label{conditionphirho}
 \phi (0) >0 ,\; \phi^{\mathrm{(odd)}} (0) =0 ,\; \rho (0) >0 ,\; \rho^{\mathrm{(odd)}} (0) =0 ,
 \end{eqnarray}
and
\begin{eqnarray} \label{conditionvarphi}
\varphi (0) =0 ,\; \varphi' (0)= 1 ,\; \varphi^{\mathrm{(even)}} (0) =0 .
\end{eqnarray}
Then we can define a Riemannian metric on $ \mathbb{R}^{m+1} \times \sphere^n \times \sphere^2 $ by
$$ g_{\varphi , \phi ,\rho } (r) = dr^2 + \varphi (r)^2 g_{\sphere^m} + \phi (r)^2 g_{\sphere^n} + \rho (r)^2 g_{\sphere^2}  .$$

See also \cite[Proposition 1.4.7]{MR3469435} for more details.

Write $X_0 =\frac{\partial }{\partial r} $, $X_1 \in T{\sphere^m} $, $X_2 \in T{\sphere^n} $, and $X_3 \in T{\sphere^2} $. Then the Ricci curvature of $g_{\varphi ,\phi ,\rho}$ can be expressed as following.

\begin{lemma}\label{lmm-riccicurvature}
Let $\varphi ,\phi , \rho $ and $g_{\varphi , \phi ,\rho }$ be as above. Then the Ricci curvature tensors of $g_{\varphi , \phi ,\rho } $ can be determined by
\begin{eqnarray}
\Ric_{g_{\varphi , \phi ,\rho } } \left(X_0 \right) & = & - \left( m \frac{\varphi''}{\varphi} + n \frac{\phi''}{\phi} + 2 \frac{ \rho''}{\rho } \right) X_0 ,\\
\Ric_{g_{\varphi , \phi ,\rho } } \left(X_1 \right) & = & \left[-  \frac{\varphi''}{\varphi} +(m-1) \frac{1-(\varphi')^2 }{\varphi^2} - n \frac{\varphi' \phi'}{\varphi \phi } - 2 \frac{\varphi' \rho'}{\varphi \rho } \right] X_1 ,\\
\Ric_{g_{\varphi , \phi ,\rho } } \left(X_2 \right) & = & \left[-  \frac{\phi''}{\phi} +(n-1) \frac{1-(\phi')^2 }{\phi^2} - m \frac{\varphi' \phi'}{\varphi \phi } - 2 \frac{\phi' \rho'}{\phi \rho } \right] X_2 ,\\
\Ric_{g_{\varphi , \phi ,\rho } } \left(X_3 \right) & = & \left[-  \frac{\rho''}{\rho} + \frac{1-(\rho')^2 }{\rho^2} - m \frac{\varphi' \rho'}{\varphi \rho } - n \frac{\phi' \rho'}{\phi \rho } \right] X_3 .
\end{eqnarray}
\end{lemma}

\begin{proof}
One can conclude it by a straightforward calculation. See also \cite[Subsection 4.2.4]{MR3469435}.
\end{proof}

\begin{lemma}\label{lem-linear-Ricci}
Define a Riemannian metric on $ (0,\infty)\times\sphere^m \times \sphere^n \times \sphere^2 $ by
\begin{align*}
g_{\varphi , \phi ,\rho } (r) = dr^2 + \varphi (r)^2 g_{\sphere^m} + \phi (r)^2 g_{\sphere^n} + \rho^2(r) g_{\sphere^2}.
\end{align*}
If $\varphi(r)=a_1r+b_1$ and $\phi(r)=\rho(r)=a_2r+b_2$,
then the Ricci curvature is
\begin{eqnarray*}
&\Ric(g_1)_{00} &=0
,\\
&\Ric(g_1)_{11} &=
\frac{[(m-1)-(m+n+1)a_1^2]a_2r
+(m-1)b_2-[(m-1)a_1b_2+(n+2)a_2b_1]a_1}
{(a_1r+b_1)^2(a_2r+b_2)}
,\\
&\Ric(g_1)_{22} & =
\frac{[(n-1)-(m+n+1)a_2^2]a_1r
+(n-1)b_1-[(n+1)a_2b_1+ma_1b_2]a_2}
{(a_1r+b_1)(a_2r+b_2)^2}
,\\
&\Ric(g_1)_{33} & =
\frac{[1-(m+n+1)a_2^2]a_1r
+b_1-[(n+1)a_2b_1+ma_1b_2]a_2}
{(a_1r+b_1)(a_2r+b_2)^2}
\end{eqnarray*}
\end{lemma}
\begin{proof}
    It follows from Lemma \ref{lmm-riccicurvature}.
\end{proof}
\section{Construction of the local model spaces}

In this section, we construct some local model spaces by triple warped products. For the remainder of the paper, we will implicitly assume that whenever a constraint on a parameter is stated, e.g. $\delta_1\leqs \delta_1(m,n,\delta)$, a possibly stronger constraint of the same form holds in the sequel whenever 
needed.

\subsection{Model I}

At first, we construct a metric $g_{\varphi ,\phi ,\rho} $ on $ (0,\infty )\times \sphere^{m} \times \sphere^{n} \times \sphere^{2} $ such that the asymptotic cone is $C(\sphere^m_{k } \times \sphere^n_{k } )$, and the topology near $r=0$ is homeomorphic to $\mathbb{R}^{m+1} \times \sphere^n \times \sphere^2 $.

\begin{lemma}\label{lmm-localmodel1}
Let $ m ,n \geqs 2 $. For any $0< \epsilon \leqs \frac{1}{100} $, $0<\delta \leqs \delta_0 (m,n,\epsilon) $ and $0<k<k_0(m,n)$, there exist constants $R(m,n,\epsilon,\delta,k)>0$ and positive functions $ \varphi ,\phi ,\rho $ on $(0,\infty )$, such that
\begin{equation*}
    \begin{array}{lll}
       \varphi |_{(0,1 )} = (1-\epsilon)r, & \phi |_{(0,1)} = \delta, &  \rho'|_{(0,1)}=0,\\
       \varphi|_{[R,+\infty )} = kr,  & \phi|_{[R,+\infty )} =kr, &\rho'|_{[R,+\infty)} =0,
    \end{array}
\end{equation*}
and $\Ric_{g_{\varphi ,\phi ,\rho}} \geqs 0$.
\end{lemma}

\begin{proof}
	We denote $M=\mathbb{R}^{m+1} \times \sphere^n \times \sphere^2$.
Let us begin with the initial metric
\begin{align*}
g_0(r) = dr^2 
+(1-\epsilon)^2r^2g_{\sphere^m} 
+\delta^2 g_{\sphere^n} 
+\delta^2 g_{\sphere^2},
\end{align*}

\textbf{Step1: Constructing $g_1$.}
Set $U_1=\{r\leqs 2\}$ and $R_1=100$.
We will define a metric $g_1$ by modifying $g_0$ on $M\backslash U_1$ through the ansatz
\begin{align*}
g_1(r) = dr^2 
+\varphi(r)^2 g_{\sphere^m} 
+\phi(r)^2 g_{\sphere^n} 
+\rho(r)^2 g_{\sphere^2}.
\end{align*}
Let $0<\delta_1(m,n,\delta,k)<\frac{1}{2}k<1$ be a constant to be determined later.
By smoothing out the function $\max\{\delta,\delta+\delta_1(r-R_1)\}$ near $r=R_1$, we can build a smooth function $\phi(r)$
\begin{equation*}
\phi(r)=
\begin{cases}
\delta 
&\textrm{if }\quad r\leqs10^{-1}R_1
\\
0<R_1 \phi''<\delta_1
&\textrm{if }\quad 10^{-1}R_1\leqs r\leqs10R_1
\\
\delta+\delta_1(r-R_1)
&\textrm{if }\quad 10R_1\leqs r.
\end{cases}
\end{equation*}
By smoothing out the function $\max\{\delta,\delta+\delta_1(r-R_1)\}$ near $r=R_1$, we can build a smooth function $\rho(r)$ satisfying
\begin{equation*}
\rho(r)=
\begin{cases}
\delta 
&\textrm{if }\quad r\leqs10^{-1}R_1
\\
0<R_1 \rho''<\delta_1 
&\textrm{if }\quad 10^{-1}R_1\leqs r\leqs10R_1
\\
\delta+\delta_1 (r-R_1)
&\textrm{if }\quad 10R_1\leqs r,
\end{cases}
\end{equation*}
respectively.
Similarly, by smoothing out the function $\min\{(1-\epsilon)r,(1-\epsilon)R_1+k(r-R_1)\}$ near $r=R_1$, we can build a smooth function $\varphi(r)$ satisfying
\begin{equation*}
\varphi(r)=
\begin{cases}
(1-\epsilon)r 
&\textrm{if }\quad r\leqs 20^{-1}R_1
\\
\varphi''<0 
&\textrm{if }\quad 
20^{-1}R_1\leqs r\leqs 10^{-1}R_1
\\
R_1 \varphi''<-\frac{1-\epsilon-k}{100}
&\textrm{if }\quad 10^{-1}R_1\leqs r\leqs10R_1
\\
(1-\epsilon)R_1+k(r-R_1) &\textrm{if }\quad 10R_1\leqs r
\end{cases}
\end{equation*}

For $r\leqs 10^{-1}R_1$, we have $k\leqs\varphi'\leqs1-\epsilon$ since $\varphi''\leqs0$.
And then
\begin{eqnarray*}
&\Ric(g_1)_{00} &= - m \frac{\varphi''}{\varphi} \geqs 0
,\\
&\Ric(g_1)_{11} &\geqs -\frac{\varphi''}{\varphi} +(m-1) \frac{\epsilon(2-\epsilon)}{\varphi^2} \geqs 0
,\\
&\Ric(g_1)_{22} & =\frac{n-1}{\delta^2}\geqs 0
,\\
&\Ric(g_1)_{33} & =\frac{1}{\delta^2}\geqs 0,
\end{eqnarray*}
where $\Ric(g_1)_{ii}=\Ric_{g_1}(X_i,X_i)$ defined in Lemma \ref{lmm-riccicurvature}.

To estimate the Ricci curvature in the interval $[R_1/10,10R_1]$, we will use the following facts
\begin{align*}
0<\rho''<\frac{\delta_1}{R_1}
,\quad
0\leqs\rho'\leqs\delta_1
,\quad
\delta\leqs\rho\leqs\delta+9\delta_1R_1
,
\end{align*}

\begin{align*}
0<\phi''<\frac{\delta_1}{R_1}
,\quad
0\leqs\phi'\leqs \delta_1
,\quad
\delta\leqs\phi\leqs\delta+9\delta_1 R_1
,
\end{align*}

\begin{align*}
\frac{1-\epsilon-k}{100R_1}<-\varphi''
,\quad
k\leqs\varphi'\leqs1-\epsilon
,\quad
\frac{(1-\epsilon)R_1}{20}\leqs\varphi\leqs(1-\epsilon+9k)R_1
.
\end{align*}

Then we have
\begin{eqnarray*}
&\Ric(g_1)_{00} &\geqs
\frac{1}{R_1}
\left[
\frac{m(1-\epsilon-k)}{100R_1(1-\epsilon+9k)}
-\frac{(n+2)\delta_1}{\delta}
\right]
,\\
&\Ric(g_1)_{11} &\geqs
\frac{1}{R_1}
\left[
\frac{1-\epsilon-k}{100R_1(1-\epsilon+9k)}
+\frac{(m-1)\epsilon(2-\epsilon)}{R_1(1-\epsilon+9k)^2}
-\frac{20(n+2)\delta_1}{(1-\epsilon)\delta}
\right]
,\\
&\Ric(g_1)_{22} &\geqs
\frac{(n-1)(1-\delta_1^2)}{(\delta+9\delta_1R_1)^2}
-\frac{\delta_1}{R_1\delta}
(20m+1+2\frac{R_1\delta_1}{\delta})
,\\
&\Ric(g_1)_{33} &\geqs
\frac{1-\delta_1^2}{(\delta+9\delta_1R_1)^2}
-\frac{\delta_1}{R_1\delta}
(20m+1+n\frac{R_1\delta_1}{\delta})
.
\end{eqnarray*}
If $0<k<10^{-2}$, $\delta_1\leqslant \delta_1(m,n,\delta)$, then we have $\Ric\geqs0$ in $[10^{-1}R_1,10R_1]$.

Apply Lemma \ref{lem-linear-Ricci}, where $a_1=k$, $a_2=\delta_1$, $b_1=R_1(1-\epsilon-k)$, $b_2=\delta-R_1\delta_1$, then we know that the Ricci curvature is non-negative for all $r>0$ if $k\leqs k(m,n)$, $\delta_1\leqs\delta_1(m,n,\delta)$.

Now we build a metric $g_1$ satisfying the initial condition we stated and have the property that
\begin{align*}
g_1(r) = dr^2 
+\left[kr+R_1(1-\epsilon-k) \right]^2
g_{\sphere^m} 
+\left(\delta_1r+\delta-R_1\delta_1 \right)^2 g_{\sphere^n} 
+\left(\delta_1r+\delta-R_1\delta_1 \right)^2 g_{\sphere^2},
\end{align*}
for $r\geqs10R_1$.

\textbf{Step2: Constructing $g_2$.}
Set $R_2=10^3R_1$, and $U_2=\{r\leqs 10 R_1\}$. We will define a metric $g_2$ by modifying $g_1$ on $M\backslash U_2$ through the ansatz
\begin{align*}
g_2(r) = dr^2 
+\left[kr+R_1(1-\epsilon-k) \right]^2
g_{\sphere^m} 
+\left(\delta_1r+\delta-R_1\delta_1 \right)^2 g_{\sphere^n} 
+\rho(r)^2 g_{\sphere^2},
\end{align*}

For $0<s<<1$ to be determined later, we consider $\rho(r)$ by smoothing the function $\min\{\delta_1r+\delta-R_1\delta_1,(\delta_1R_2+\delta-R_1\delta_1)(\frac{r}{R_2})^s)\}$ at $R_2$  with the following properties
\begin{equation*}
\rho(r)=
	\begin{cases}
		\delta_1r+\delta-R_1\delta_1 & \text{if } r\leqs10^{-1}R_2,\\
		\rho''\leqs0 & \text{if }10^{-1}R_2\leqs r\leqs 10R_2,\\
		ar^s & \text{if }r\geqs 10R_2,
	\end{cases}
\end{equation*}
where $a=(\delta_1R_2+\delta-R_1\delta_1)R_2^{-s}$.

For $r\in [10^{-1}R_2,10R_2]$, we have
\begin{align*}
	\rho''\leqs0,\quad  sa(10R_2)^{s-1}\leqs \rho'\leqs \delta_1,\quad \delta\leqs\rho\leqs 20\delta.
\end{align*}
Then we have
\begin{align*}
	\Ric(g_2)_{00}&=-2\frac{\rho''}{\rho}\geqs0,\\
	\Ric(g_2)_{11}&\geqs (m-1)\frac{1-k^2}{\varphi^2}-(n+2)\frac{k\delta_1}{\delta\varphi}\geqs\varphi^{-2}\left[\frac{m-1}{2}-k(n+2)(10kR_2+R_1)\frac{\delta_1}{\delta}\right],\\
	\Ric(g_2)_{22}&\geqs (n-1)\frac{1-\delta_1^2}{\phi^2}-\frac{mk\delta_1}{\phi}-\frac{2\delta_1^2}{\delta\phi}\geqs \phi^{-2}\left[\frac{n-1}{2}-(mk+2\frac{\delta_1}{\delta})\delta_1(10\delta_1R_2+\delta)\right],\\
	\Ric(g_2)_{33}&\geqs \frac{1-\delta_1^2}{\rho^2}-\frac{mk\delta_1}{\rho}-\frac{n\delta_1^2}{\delta\rho}\geqs \rho^{-2}\left[\frac12-20(mk\delta_1\delta+n\delta_1^2)\right].
\end{align*}
So $\Ric\geqs0$ for $r\in [10^{-1}R_2,10R_2]$ if $\delta_1/\delta\leqs c(m,n)$.

For $r\geqs 10R_2$,
\begin{align*}
	\Ric(g_2)_{00}&=-2\frac{\rho''}{\rho}\geqs0,\\
	\Ric(g_2)_{11}&\geqs (m-1)\frac{1-k^2}{\varphi^2}-\frac{nk}{r\varphi}-\frac{2sk}{r\varphi}\geqs (r\varphi^2)^{-1}\left[\frac{m-1}{2}r-k(n+2s)(kr+R_1)\right],\\
	\Ric(g_2)_{22}&\geqs (n-1)\frac{1-\delta_1^2}{\phi^2}-\frac{m\delta_1}{r\phi}-\frac{2s\delta_1}{r\phi}\geqs (r\phi^2)^{-1}\left[\frac{n-1}{2}r-\delta_1(m+2s)(\delta_1r+\delta)\right],\\
	\Ric(g_2)_{33}&\geqs \frac{s(1-s)}{r^2}+\frac{1-\delta_1^2}{a^2r^{2s}}-\frac{ms}{r^2}-\frac{ns}{r^2}\geqs \frac{R_2^{2-2s}}{2a^2r^2}-\frac{(m+n)s}{r^2}.\end{align*}
If $0<s<s(m,n)$, $\Ric(g_2)\geqs0$ for all $r>0$.

Now we build a metric $g_2$ satisfying the initial condition we stated and have the property that
\begin{align*}
g_2(r) = dr^2 
+\left[kr+R_1(1-\epsilon-k) \right]^2
g_{\sphere^m} 
+\left(\delta_1r+\delta-R_1\delta_1 \right)^2 g_{\sphere^n} 
+\left(ar^s\right)^2 g_{\sphere^2},
\end{align*}
for $r\geqs10R_2$.

\textbf{Step3: Constructing $g_3$.}
Set $U_3=\{r\leqs 10 R_2\}$. We will define a metric $g_3$ by modifying $g_2$ on $M\backslash U_3$ through the ansatz
\begin{align*}
g_3(r) = dr^2 
+\varphi(r)^2
g_{\sphere^m} 
+\phi(r)^2 g_{\sphere^n} 
+\left(ar^s\right)^2 g_{\sphere^2}.
\end{align*}

For $R_3=R_3(m,n,\delta,\delta_1,k,s)$, $R_4=\Lambda R_3$, $\Lambda=\Lambda(R_3)$, we will choose smooth functions $\varphi(r)$ and $\phi(r)$ satisfying



\begin{equation*}
\varphi(r)=
\begin{cases}
kr+b_1
&\textrm{if }\quad r\leqs R_3
\\
kr\leqs \varphi\leqs kr+b_1, |\varphi'|<2k, r\varphi''< (\ln R_3)\iv
&\textrm{if }\quad R_3\leqs r\leqs R_4
\\
kr
&\textrm{if }\quad R_4\leqs r.
\end{cases}
\end{equation*}

\begin{equation*}
\phi(r)=
\begin{cases}
\delta_1r+b_2
&\textrm{if }\quad r\leqs R_3
\\
\delta_1 r\leqs\phi\leqs kr, |\phi'|<2k, r\phi''<(\ln R_3)\iv, |\phi'/\phi|<10 r\iv
&\textrm{if }\quad R_3\leqs r\leqs R_4
\\
kr
&\textrm{if }\quad R_4\leqs r,
\end{cases}
\end{equation*}
respectively, where $b_1=R_1(1-\epsilon-k)$, $b_2=\delta-R_1\delta_1$.

For the existence of $\varphi(r)$, we can smooth the continuous function
\begin{equation*}
\hat{\varphi}(r)=
\begin{cases}
kr+b_1
&\textrm{if }\quad r\leqs R_3
\\
\max\{kr-c\left(r\ln r-(\ln R_3+1)r\right)+b_3,kr\}
&\textrm{if }\quad R_3\leqs r,
\end{cases}
\end{equation*}
where $b_3$ is the constant such that $\hat{\varphi}$ is continuous at $R_3$, and $c=(10\ln R_3)\iv$. Let $f:=-c\left(r\ln r-(\ln R_3+1)r\right)+b_3$, then we have $f(R_3)=b_1\leqs100$, $f'(R_3)=0$, $rf''=-c$. If $f(\lambda R_3)=0$, then
\begin{align*}
    c\left(\lambda R_3\ln(\lambda R_3)-(\ln R_3+1)\lambda R_3\right)-c(R_3\ln R_3-(\ln R_3+1)R_3)=b_1,
\end{align*}
which implies
\begin{align*}
    (\lambda\ln \lambda-\lambda+1)\cdot \frac{R_3}{10\ln R_3}=b_1.
\end{align*}
Since $R_3\geqs R_2=10^5$, we have $\lambda\leqs 10$. Moreover,
\begin{align*}
    |f'|\leqs c(\ln (\lambda R_3)-\ln R_3)\leqs (\ln R_3)\iv.
\end{align*}
 So if $\Lambda\geqs 20$, $R_3\geqs R_3(k)$, we can get the desired $\varphi$ by smoothing.

 For the existence of $\phi(r)$, we can smooth the continuous function
 \begin{equation*}
\hat{\phi}(r)=
\begin{cases}
\delta_1r+b_2
&\textrm{if }\quad r\leqs R_3
\\
\delta_1 r+c(r\ln r-(\ln R_3+1)r)+b_4
&\textrm{if }\quad R_3\leqs r\leqs \lambda_1R_3
\\
kr
&\textrm{if }\quad \lambda_1R_3\leqs r,
\end{cases}
\end{equation*}
where $b_4$ is the constant such that $\hat{\phi}$ is continuous at $R_3$, and $\lambda_1$ is the constant such that $\hat{\phi}$ is continuous at $\lambda_1R_3$. So by definition, we have
\begin{align*}
    \delta_1 \lambda_1R_3+c(\lambda_1R_3\ln (\lambda_1R_3)-(\ln R_3+1)\lambda_1R_3)+b_4=k\lambda_1R_3,
\end{align*}
which implies that
\begin{align*}
    c(\ln\lambda_1-1+\lambda_1\iv)+\frac{b_2}{\lambda_1R_3}=k-\delta_1.
\end{align*}
Then $c\ln\lambda_1<c+k-\delta_1<k$ if we choose $R_3\geqs R_3(\delta_1)$. On the other hand, we have
\begin{align*}
    c\ln\lambda_1=k-\delta_1+\frac{cR_3(\lambda_1-1)-b_2}{\lambda_1R_3}>k-\delta_1.
\end{align*}
Thus, 
\begin{align*}
    \hat{\phi}'_-(\lambda_1R_3)=c\ln\lambda_1+\delta_1\in(k,k+\delta_1),\quad \lambda_1\leqs R_3^{10k}\leqs \sqrt{R_3}.
\end{align*}
The last condition follows from
\begin{align*}
    \frac{\hat{\phi}'}{\hat{\phi}}&=\frac{c(\ln r-\ln R_3)+\delta_1}{c(r\ln r-(\ln R_3+1)r)+\delta_1r+b_4}=r\iv\cdot\frac{1}{1+b_4(r\hat{\phi}')\iv-c(\hat{\phi}')\iv}\\
    &\leqs r\iv \cdot \frac{1}{1-c\delta_1\iv}.
\end{align*}
So if $\Lambda\geqs \sqrt{R_3}$, $R_3\geqs R_3(\delta_1,k)$, there exists $\phi(r)$ as claimed.

For $r\in [R_3,R_4]$,
\begin{align*}
	\Ric(g_3)_{00}&\geqs\frac{2s(1-s)}{r^2}-\frac{m(\ln R_3)\iv}{kr^2}-\frac{n(\ln R_3)\iv}{\delta_1 r^2},\\
	\Ric(g_3)_{11}&\geqs (m-1)\frac{1-4k^2}{\varphi^2}-\frac{(\ln R_3)\iv}{r\varphi}-\frac{20nk}{ r\varphi}-\frac{4ks}{r\varphi}\geqs (r\varphi^2)\iv\left[\frac{m-1}{2}r-100nk(kr+b_1)\right],\\
	\Ric(g_3)_{22}&\geqs (n-1)\frac{1-4k^2}{\phi^2}-\frac{(\ln R_3)\iv}{r\phi}-\frac{20mk}{ r\phi}-\frac{4ks}{r\phi}\geqs (r\phi^2)\iv\left[\frac{n-1}{2}r-100mk^2r\right],\\
	\Ric(g_3)_{33}&\geqs \frac{1-\delta_1^2}{a^2r^{2s}}-\frac{20(m+n)s}{r^2}.\end{align*}
Then $\Ric\geqs0$ for $r\in[ R_3, R_4]$ after choosing enough large $R_3$.

For $r\geqs R_4$,
\begin{align*}
	\Ric(g_3)_{00}&=\frac{2s(1-s)}{r^2}>0,\\
	\Ric(g_3)_{11}&=\frac{(m-1)(1-k^2)}{k^2r^2}-\frac{n+2s}{r^2}>0,\\
	\Ric(g_3)_{22}&=\frac{(n-1)(1-k^2)}{k^2r^2}-\frac{m+2s}{r^2}>0,\\
	\Ric(g_3)_{33}&=\frac{s(1-s)}{r^2}+\frac{1-s^2a^2r^{2s-2}}{a^2r^{2s}}-\frac{s(m+n)}{r^2}>0.\end{align*}

Now we build a metric $g_3$ with $\Ric(g_3)\geqs0$ satisfying the initial condition we stated and have the property that
\begin{align*}
g_3(r) = dr^2 
+(kr)^2
g_{\sphere^m} 
+(kr)^2 g_{\sphere^n} 
+\left(ar^s\right)^2 g_{\sphere^2},
\end{align*}
for $r\geqs R_4$.

\textbf{Step4: Constructing $g_4$.}
Set $U_4=\{r\leqs 10 R_4\}$. We will define a metric $g_4$ by modifying $g_3$ on $M\backslash U_4$ through the ansatz
\begin{align*}
g_4(r) = dr^2 
+(kr)^2
g_{\sphere^m} 
+(kr)^2 g_{\sphere^n} 
+\rho(r)^2 g_{\sphere^2}.
\end{align*}
Then the Ricci curvature of this ansatz is
\begin{align*}
	\Ric(g_4)_{00}&=-\frac{2\rho''}{\rho},\\
	\Ric(g_4)_{11}&=\frac{(m-1)(1-k^2)}{k^2r^2}-\frac{n}{r^2}-\frac{2\rho'}{r\rho},\\
	\Ric(g_4)_{22}&=\frac{(n-1)(1-k^2)}{k^2r^2}-\frac{m}{r^2}-\frac{2\rho'}{r\rho},\\
	\Ric(g_4)_{33}&=-\frac{2\rho''}{\rho}+\frac{1-\rho'^2}{\rho^2}-\frac{(m+n)\rho'}{r\rho}.
\end{align*}
We can choose $\rho(r)$ of the form
\begin{equation*}
\rho(r)=
\begin{cases}
ar^s
&\textrm{if }\quad r\leqs10R_4
\\
\rho''\leqs0
&\textrm{if }\quad 10R_4\leqs r\leqs10^3 R_4
\\
\lambda
&\textrm{if }\quad 10^3R_4\leqs r,
\end{cases}
\end{equation*}
for some $\lambda=\lambda(a,s,R_4)$. Then it's easy to see that $\Ric\geqs0$ for any $r>0$. Moreover, for $R=10^4R_4$, the last metric $g_4$ satisfies all the properties we stated.

\end{proof}

\subsection{Model II}

Next we construct a metric $g_{\varphi ,\phi ,\rho} $ on $ (0,\infty )\times \sphere^{m} \times \sphere^{n} \times \sphere^{2} $ such that the metric around $\infty $ is isometric to $ C(\sphere^m_{1-\epsilon } ) \times \sphere^n_{\delta } \times \sphere^2_{\rho } $

\begin{lemma}\label{lmm-localmodel2}
Let $ m ,n \geqs 2 $. Then for any $0< \epsilon \leqs \frac{1}{100} $, $\lambda>0$, $0<k<k_0(m,n)$, there are constants $R(m,n,k,\epsilon )>0$, $\delta (m,n,k,\epsilon) >0$ and positive functions $ \varphi ,\phi ,\rho $ on $(0,\infty )$, such that
\begin{equation*}
    \begin{array}{lll}
       \varphi |_{(0,1 )} = kr, & \phi |_{(0,1)} = kr, &  \rho|_{(0,1)}=\lambda, \\
       \varphi|_{[R,+\infty )} = (1-\epsilon ) r,  & \phi|_{[R,+\infty )} = \delta, &\rho|_{[R,+\infty)} =\lambda,
    \end{array}
\end{equation*}
and $\Ric_{g_{\varphi ,\phi ,\rho}} \geqs 0$.
\end{lemma}

\begin{proof}
	Let us begin with the initial metric
\begin{align*}
g_0(r) = dr^2 
+(kr)^2g_{\sphere^m} 
+(kr)^2 g_{\sphere^n} 
+\lambda^2 g_{\sphere^2},
\end{align*}

\textbf{Step1: Constructing $g_1$.}
Set $U_1=\{r\leqs 2\}$ and $R_1=100$.
We will define a metric $g_1$ by modifying $g_0$ on $M\backslash U_1$ through the ansatz
\begin{align*}
g_1(r) = dr^2 
+(kr)^2 g_{\sphere^m} 
+\phi(r)^2 g_{\sphere^n} 
+\lambda^2 g_{\sphere^2}.
\end{align*}

Set $s=\epsilon/(10^6mn)$.
By smoothing out the function $\min\{kr,kR_1^{1-s}r^s\}$ near $r=R_1$, we can build a smooth function $\phi(r)$ of the form
\begin{equation*}
\phi(r)=
\begin{cases}
kr
&\textrm{if }\quad r\leqs10^{-1}R_1
\\
\phi''\leqs0
&\textrm{if }\quad 10^{-1}R_1\leqs r\leqs10R_1
\\
kR_1(\frac{r}{R_1})^s
&\textrm{if }\quad 10R_1\leqs r.
\end{cases}
\end{equation*}

The Ricci curvature of the ansatz is
\begin{align*}
	\Ric(g_1)_{00}&=-\frac{n\phi''}{\phi},\\
	\Ric(g_1)_{11}&=\frac{(m-1)(1-k^2)}{k^2r^2}-\frac{n\phi'}{r\phi},\\
	\Ric(g_1)_{22}&=-\frac{\phi''}{\phi}+\frac{(n-1)(1-\phi'^2)}{\phi^2}-\frac{m\phi'}{r\phi},\\
	\Ric(g_1)_{33}&=\frac{1}{\lambda^2}>0.
\end{align*}
By direct computation, we have $\Ric(g_1)\geqs0$ for any $r>0$ if $0<k<k_0(m,n)$.

Now we build a metric $g_1$ satisfying the initial condition we stated and have the property that
\begin{align*}
g_1(r) = dr^2 
+(kr)^2
g_{\sphere^m} 
+(ar^s)^2 g_{\sphere^n} 
+\lambda^2 g_{\sphere^2},
\end{align*}
for $r\geqs10R_1$, where $a=kR_1^{1-s}$.

\textbf{Step2: Constructing $g_2$.}
Set $U_2=\{r\leqs 10 R_1\}$. We will define a metric $g_2$ by modifying $g_1$ on $M\backslash U_2$ through the ansatz
\begin{align*}
g_2(r) = dr^2 
+\varphi(r)^2
g_{\sphere^m} 
+\left(ar^s \right)^2 g_{\sphere^n} 
+\lambda^2 g_{\sphere^2},
\end{align*}

For $R_2=R_2(\epsilon)$, $R_3=e^{100(ks^2)^{-1}}R_2$, we can choose a smooth function $\varphi(r)$  with the following properties
\begin{equation*}
\varphi(r)=
	\begin{cases}
		kr & \textrm{if } r\leqs R_2,\\
		kr\leqs\varphi\leqs r, |\varphi'|\leqs (1-10^{-1}\epsilon), |r\varphi''|<ks^2 & \textrm{if }R_2\leqs r\leqs R_3,\\
		(1-\epsilon)r & \textrm{if }r\geqs R_3.
	\end{cases}
\end{equation*}
The Ricci curvature of the ansatz is
\begin{align*}
	\Ric(g_2)_{00}&=-\frac{m\varphi''}{\varphi}+n\frac{s(1-s)}{r^2},\\
	\Ric(g_2)_{11}&=-\frac{\varphi''}{\varphi}+(m-1)\frac{(1-\varphi'^2)}{\varphi^2}-\frac{ns\varphi'}{r\varphi},\\
	\Ric(g_2)_{22}&=\frac{s(1-s)}{r^2}+\frac{n-1}{a^2r^{2s}}-\frac{(n-1)s^2}{r^2}-\frac{ms\varphi'}{r\varphi},\\
	\Ric(g_2)_{33}&=\frac{1}{\lambda^2}>0.
\end{align*}
By direct computation, we have $\Ric(g_2)\geqs0$ for any $r>0$.

Now we build a metric $g_2$ satisfying the initial condition we stated and have the property that
\begin{align*}
g_2(r) = dr^2 
+\left[(1-\epsilon)r\right]^2
g_{\sphere^m} 
+(ar^s)^2 g_{\sphere^n} 
+\lambda^2 g_{\sphere^2},
\end{align*}
for $r\geqs10R_3$,

\textbf{Step3: Constructing $g_3$.}
Set $U_3=\{r\leqs 10 R_3\}$. We will define a metric $g_3$ by modifying $g_2$ on $M\backslash U_3$ through the ansatz
\begin{align*}
g_3(r) = dr^2 
+\left[(1-\epsilon)r\right]^2
g_{\sphere^m} 
+\phi(r)^2 g_{\sphere^n} 
+\lambda^2 g_{\sphere^2},
\end{align*}
Then the Ricci curvature of this ansatz is
\begin{align*}
	\Ric(g_3)_{00}&=-\frac{n\phi''}{\phi},\\
	\Ric(g_3)_{11}&=\frac{(m-1)\epsilon(2-\epsilon)}{(1-\epsilon)^2r^2}-\frac{n\phi'}{r\phi},\\
	\Ric(g_3)_{22}&=-\frac{\phi''}{\phi}+\frac{(n-1)(1-\phi'^2)}{\phi^2}-\frac{m\phi'}{r\phi},\\
	\Ric(g_3)_{33}&=\frac{1}{\lambda^2}>0.\end{align*}
We can choose $\phi(r)$ of the form
\begin{equation*}
\phi(r)=
\begin{cases}
ar^s
&\textrm{if }\quad r\leqs10R_3
\\
\phi''\leqs0
&\textrm{if }\quad 10R_3\leqs r\leqs10^3 R_3
\\
\delta
&\textrm{if }\quad 10^3R_3\leqs r,
\end{cases}
\end{equation*}
for some $\delta=\delta(a,R_2,s)$. Then by direct computation we have $\Ric(g_3)\geqs0$ for any $r>0$. Moreover, for $R=10^4R_3$, the last metric $g_3$ satisfies all the properties we stated.

\end{proof}

\section{Connecting $\rean^m$ and $\rean^n$}
We first combine the Model I and Model II into a block which is nearly identical at both ends.

\begin{lemma}\label{lmm-unit}
	For any $m,n\geqs2$, $\epsilon>0$ and $L>1$, then there exists $k=k(m,n)$, $R=R(m,n,\epsilon,L)>1$, $0<\delta<c(m,n,\epsilon)L\iv$, and positive smooth functions $\varphi,\phi,\rho$ on $r\in (0,\infty)$ such that
	\begin{equation*}
		\begin{array}{lll}
		\varphi|_{(0,(LR)\iv)}=kr, & \phi|_{(0,(LR)\iv)}=kr, & \rho|_{(0,(LR)\iv)}=\lambda_1,\\
		\varphi|_{[L\iv,1]}=(1-\epsilon)r, & \phi|_{[L\iv,1]}=\delta, & \rho|_{[L\iv,1]}=\lambda_1,\\
		\varphi|_{[R,\infty)}=kr, & \phi|_{[R,\infty)}=kr,& \rho|_{[R,\infty)}=\lambda_2,
	\end{array}
	\end{equation*}
	 and $\Ric_{\varphi,\phi,\rho}\geqs0$.
\end{lemma}

\begin{proof}
	For $m,n\geqs2$, take $k=k_0(m,n)$ to be the smaller one in lemma \ref{lmm-localmodel1} and lemma \ref{lmm-localmodel2}. By lemma \ref{lmm-localmodel1}, we have $\delta_0=\delta_0(m,n,\epsilon)$. By lemma \ref{lmm-localmodel2}, we have $\delta_2(m,n,k_0,\epsilon)$ and $R_2(m,n,k_0,\epsilon)$. After possibly increasing $R_2$, we can assume $\delta_2/R_2<\delta_0$. Then we take $\delta=\frac{\delta_2}{LR_2}<\delta_0$. Applying lemma \ref{lmm-localmodel1}, we get $R=R_1(m,n,\epsilon,\delta,k_0)>R_2$ and functions $\varphi_1$, $\phi_1$, $\rho_1$ satisfying
	\begin{equation*}
		\begin{array}{lll}
		\varphi_1|_{(0,1)}=(1-\epsilon)r, & \phi_1|_{(0,1)}=\delta, & \rho_1|_{(0,1)}=\lambda_1,\\
		\varphi_1|_{[R,\infty)}=kr, & \phi_1|_{[R,\infty)}=kr, & \rho_1|_{[R,\infty)}=\lambda_2.
	\end{array}
	\end{equation*}
		
		Applying lemma \ref{lmm-localmodel2} with $\lambda:=LR_2\lambda_1$, we get functions $\varphi_2$, $\phi_2$, $\rho_2$. We rescale the functions by $\tilde{\varphi}(r):=(LR_2)\iv \varphi(LR_2r)$. Similarly we get $\tilde{\phi}$ and $\tilde{\rho}$, then they satisfy the following
		\begin{equation*}
		\begin{array}{lll}
		\tilde{\varphi}_2|_{(0,(LR)\iv)}=kr, & \tilde{\phi}_2|_{(0,(LR)\iv)}=kr, & \tilde{\rho}_2|_{(0,(LR)\iv)}=\lambda_1,\\
		\tilde{\varphi}_2|_{[L\iv,\infty)}=(1-\epsilon)r, & \tilde{\phi}_2|_{[L\iv,\infty)}=(LR_2)\iv \delta_2, & \tilde{\rho}_2|_{[L\iv,\infty)}=\lambda_1.
	\end{array}
	\end{equation*}
	Note that since $(LR_2)\iv\delta_2=\delta$, two groups of functions agree in $r\in[L\iv,1]$ respectively. Then we can glue them to get the new functions $\varphi$, $\phi$, $\rho$. These functions satisfy all the properties we stated.

\end{proof}

Now we can connect $\rean^m$ and $\rean^n$ by gluing two blocks and exchanging $m$ and $n$.

\begin{proposition}\label{prop-connecting}
	For any $m,n\geqs2$, $\epsilon>0$ and $L>1$, then there exists $k=k(m,n)$, $R=R(m,n,\epsilon,L)>1$ and positive smooth functions $\varphi,\phi,\rho$ on $r\in (0,\infty)$ such that
	\begin{equation*}
		\begin{array}{lll}
		\varphi|_{(0,(2L^2R^3)\iv)}=kr, & \phi|_{(0,(2L^2R^3)\iv)}=kr, & \rho|_{(0,(2L^2R^3)\iv)}=\lambda_1,\\
		\varphi|_{[(2L^2R^2)\iv,(2LR^2)\iv]}=\delta_1, & \phi|_{[(2L^2R^2)\iv,(2LR^2)\iv]}=(1-\epsilon)r, & \rho|_{[(2L^2R^2)\iv,(2LR^2)\iv]}=\lambda_1,\\
		\varphi|_{[(2LR)\iv,(LR)\iv]}=kr, & \phi|_{[(2LR)\iv,(LR)\iv]}=kr, & \rho|_{[(2LR)\iv,(LR)\iv]}=\lambda_2,\\
		\varphi|_{[L\iv,1]}=(1-\epsilon)r, & \phi|_{[L\iv,1]}=\delta_2, & \rho|_{[L\iv,1]}=\lambda_2\\
		\varphi|_{[R,\infty)}=kr, & \phi|_{[R,\infty)}=kr,& \rho|_{[R,\infty)}=\lambda_3,
	\end{array}
	\end{equation*}
	 where $0<\delta_1<c(m,n,\epsilon)(L^2R^2)\iv$, $0<\delta_2<c(m,n,\epsilon)L\iv$ and $\Ric_{\varphi,\phi,\rho}\geqs0$.

\end{proposition}

\begin{proof}
	First apply Lemma \ref{lmm-unit} to get $\varphi_1$, $\phi_1$, $\rho_1$. Next we exchange $m$ and $n$ and then apply Lemma \ref{lmm-unit} again to get $\varphi_2$, $\phi_2$, $\rho_2$. Rescale the second metric $\tilde{\varphi}(r)=(2LR^2)\iv\varphi(2LR^2 r)$, $\tilde{\phi}$ and $\tilde{\rho}$ likewise. Then two metrics agree on $r\in [(2LR)\iv,(LR)\iv]$. We can glue them to get the desired functions.
\end{proof}

\section{Proof of Theorem \ref{thm-weakregular}}

First we give a smoothing lemma to construct smooth metrics at the origin after slightly adjusting the position of the origin.

\begin{lemma}\label{lmm-smoothing}
    For any $m,n\geqs2$, $0<\epsilon<10\iv$ and $L>2$, $\varphi|_{[L\iv,1]}=(1-\epsilon)r$, $\phi|_{[L\iv,1]}=\delta$, $\rho|_{[L\iv,1]}=\lambda$, $\Ric_{\varphi,\phi,\rho}\geqs0$, then we can take smooth modified functions $\hat{\varphi}$, $\hat{\phi}$, $\hat{\rho}$ on $[\epsilon L\iv,+\infty)$ such that
    \begin{equation*}
		\begin{array}{lll}
		\hat{\varphi}|_{[\epsilon L\iv,2\epsilon L\iv)}=r-\epsilon L\iv, & \hat{\phi}|_{[\epsilon L\iv,2\epsilon L\iv)}=\delta, & \hat{\rho}|_{[\epsilon L\iv,2\epsilon L\iv)}=\lambda,\\
		\hat{\varphi}|_{[2L\iv,\infty)}=\varphi, & \hat{\phi}|_{[2L\iv,\infty)}=\phi, & \hat{\rho}|_{[2L\iv,\infty)}=\rho,
	\end{array}
	\end{equation*}
	and $\Ric_{\hat{\varphi},\hat{\phi},\hat{\rho}}\geqs0$.
\end{lemma}

\begin{proof}
	We construct $\hat{g}$ by modifying $g$ on $r\in (\epsilon L\iv,2L\iv)$ through the ansatz
	\begin{align*}
		\hat{g}=dr^2+\varphi(r)^2g_{\sphere^m}+\delta^2 g_{\sphere^n}+\lambda^2 g_{\sphere^2}.
	\end{align*}
	The Ricci curvature of this ansatz is
	\begin{align*}
	\Ric(\hat{g})_{00}&=-\frac{m\varphi''}{\varphi},\\
	\Ric(\hat{g})_{11}&=-\frac{\varphi''}{\varphi}+\frac{(m-1)(1-\varphi'^2)}{\varphi^2},\\
	\Ric(\hat{g})_{22}&=\frac{n-1}{\delta^2}>0,\\
	\Ric(\hat{g})_{33}&=\frac{1}{\lambda^2}>0.\end{align*}
	So we can choose smooth $\varphi(r)$ such that $\Ric(\hat{g})\geqs0$ of the form
	\begin{equation*}
	\varphi(r)=
	\begin{cases}
	r-\epsilon L\iv
	&\textrm{if }\quad \epsilon L\iv\leqs r\leqs 2\epsilon L\iv,
	\\
	\varphi''\leqs0
	&\textrm{if }\quad 2\epsilon L\iv \leqs r \leqs 2L\iv ,
	\\
	(1-\epsilon)r
	&\textrm{if }\quad 2L\iv \leqs r\leqs 1.
	\end{cases}
	\end{equation*}

\end{proof}

Now we are ready to prove Theorem \ref{thm-weakregular}.

\begin{pfthm}
	Let $m\geqs n\geqs 2$ be integers. For any $i\geqs1$, we will construct smooth metrics $g_i=(\varphi_i,\phi_i,\rho_i)$ with $\Ric(g_i)\geqs0$ on $M=(0,\infty)\times \sphere^m\times\sphere^n\times \sphere^2$. Moreover, we will find a sequence of numbers $N_i\geqs10 N_{i-1}$ such that $(\varphi_i,\phi_{i})$ and $(\varphi_{i+1},\phi_{i+1})$ coincide on $r\in [N_i\iv,\infty)$ and $\rho_i\leqs N_i^{-10}$.
	
	Set $\epsilon_i=100^{-i}$. Apply Proposition \ref{prop-connecting} with $\epsilon=\epsilon_1$ and $L_1=10$, then we get $g_1=(\varphi_1,\phi_1,\rho_1)$ and $R_1$. Set $N_1=2L_1^2R_1^3$. We also have $\rho_1\leqs N_1^{-4}$ after possibly scaling $N_1^{-4}\rho_1(r)$.  Note that the Ricci curvature will increase if we change $\rho(r)$ into $N\iv\rho(r)$. 
	
	We construct $g_{i+1}$ by induction. Assume we have already constructed $g_i$ and $N_i$, and $\varphi_i(r)=\phi_i(r)=kr$, $\rho(r)=\lambda_i$ on $r\in(0,N_i\iv)$. Again apply Proposition \ref{prop-connecting} with $\epsilon_{i+1}$ and $L_{i+1}=10^{i+1}$, then we get $(\tilde{\varphi}_{i+1},\tilde{\phi}_{i+1},\tilde{\rho}_{i+1})$ and $R_{i+1}$. After scaling $\varphi(r)=(2N_iR_{i+1})\iv\tilde{\varphi}_{i+1}(2N_iR_{i+1}r)$, $(\tilde{\varphi}_{i+1},\tilde{\phi}_{i+1})$ agree with $(\varphi_i,\phi_i)$ on $r\in [(2N_i)\iv,N_i\iv]$. Although $\tilde{\rho}_{i+1}$ may not agree with $\rho_i$, we can make them equal by scaling both. So we can glue them to get $g_{i+1}$. Set $N_{i+1}=4N_iL_{i+1}^2R_{i+1}^4$, then $\rho_{i+1}\leqs N_{i+1}^{-4}$ after possible scaling.
	
	Next we modify $g_i$ on $r\in (\epsilon_iR_iN_i\iv,2R_iN_i\iv)$ by Lemma \ref{lmm-smoothing}, then we get $\hat{g}_i$, which is also smooth at the adjusted origin $r=\epsilon_iR_iN_i\iv$. We denote this origin by $O_i$. Set $\hat{M}=\rean^{n+1}\times\sphere^m\times\sphere^2$. Now $(\hat{M}^{m+n+3},\hat{g}_i,O_i)$ are a sequence of pointed complete smooth metric with $\Ric(\hat{g}_i)\geqs0$, then by Gromov's precompactness theorem, up to subsequence, there exists a metric space $(X,d)$ such that
	\begin{align*}
		(\hat{M},\hat{g}_i,O_i)\stackrel{GH}{\longrightarrow} (X,d ,p ).
	\end{align*}
	
	On one hand, first note that for $A_i:=N_i R_i\iv L_i^{-1/2}\to\infty$, the rescaled metrics $(\hat{M},A_i^2\hat{g}_j,O_j)$ for $j\geqs i$ become
	\begin{align*}
		\varphi_j|_{(L_i^{-1/2},\frac12L_i^{1/2})}\leqs c L_i^{-1/2},\quad
		\phi_j|_{(L_i^{-1/2},\frac12L_i^{1/2})}=(1-\epsilon_i)r,\quad
		\rho_j|_{(L_i^{-1/2},\frac12L_i^{1/2})}\leqs N_i^{-5}.
	\end{align*}
	
	Let $j\to\infty$ and denote $A_{a,b}(X,d,p):=\{x\in X:a<d(x,p)<b\}$, then
	\begin{align*}
		d_{GH}\left(A_{L_i^{-1/2},\frac12L_i^{1/2}}(X,A_id,p),A_{L_i^{-1/2},\frac12L_i^{1/2}}(\rean^{n+1},g_0,0^{n+1})\right)\leqs \Psi(i\iv).
	\end{align*}
	Then let $i\to\infty$, we have
	\begin{align*}
		(X,A_id,p)\stackrel{GH}{\longrightarrow} (\rean^{n+1},g_0,0^{n+1}).
	\end{align*}
	
	On the other hand, note that for $B_i:=\frac12 N_i R_i^{-3} L_i^{-3/2}\to\infty$, the rescaled metrics $(\hat{M},B_i^2\hat{g}_j,O_j)$ for $j\geqs i$ become
	\begin{align*}
		\varphi_j|_{(L_i^{-1/2},\frac12L_i^{1/2})}=(1-\epsilon_i)r,\quad
		\phi_j|_{(L_i^{-1/2},\frac12L_i^{1/2})}\leqs c L_i^{-1/2},\quad
		\rho_j|_{(L_i^{-1/2},\frac12L_i^{1/2})}\leqs N_i^{-5}.
	\end{align*}
	Then
	\begin{align*}
		d_{GH}\left(A_{L_i^{-1/2},\frac12L_i^{1/2}}(X,B_id,p),A_{L_i^{-1/2},\frac12L_i^{1/2}}(\rean^{m+1},g_0,0^{m+1})\right)\leqs \Psi(i\iv).
	\end{align*}
	Then let $i\to\infty$, we have
	\begin{align*}
		(X,B_id,p)\stackrel{GH}{\longrightarrow} (\rean^{m+1},g_0,0^{m+1}).
	\end{align*}

\end{pfthm}

\section{Proof of Theorem \ref{thm-riemcone}}
For the convenience of the readers, we first recall the definition of Riemannian cones.
\begin{definition}
    Given a closed Riemannian manifold $(M^n,g)$, then the \emph{Riemannian cone} $C(M^n,g)$ over $(M^n,g)$ is a metric completion of the following warped product
    \begin{align*}
        \tilde{M}:=\rean_+\times M,\quad \tilde{g}:=dr^2+r^2g.
    \end{align*}
\end{definition}

Next we compute the Ricci curvature of the $N$-fold warped product.
\begin{lemma}\label{lmm-nwarped}
    Given a finite collection $(M_i^{n_i},g_i)$ $(i=1,\cdots,N)$ of closed Riemannian manifolds. Let $\hat{M}:=\rean_+\times \left(\prod_{i=1}^N M_i\right)$ be the $N$-fold warped product, i.e. 
    \begin{align*}
        \hat{g}=dr^2+\sum_{i=1}^N f_i^2(r)g_i.
    \end{align*}
    Then the Ricci curvatures of $\hat{g}$ are
    \begin{align*}
        \Ric(X_0)&=\left(-\sum_{i=1}^N n_i\frac{f_i''}{f_i}\right)X_0,\\
        \Ric(X^{(i)}_j)&=\left[-\frac{f_i''}{f_i}+\frac{\Ric^{(i)}_j-(n_i-1)(f_i')^2}{f_i^2}-\sum_{l\neq i}n_l\frac{f_i'f_l'}{f_if_l}\right]X^{(i)}_j,
    \end{align*}
    where $X_0=\frac{\dd}{\dd r}$, $X^{(i)}_j\in T_{M_i} (j=1,\cdots,n_i)$ and $\Ric_{g_i}(X^{(i)}_j)=\Ric^{(i)}_jX^{(i)}_j$.
\end{lemma}
\begin{proof}
    By direct computation.
\end{proof}
\begin{lemma}
    Under the assumption of Lemma \ref{lmm-nwarped}. Assume $\Ric_{g_i}\geqs (n_i-1)g_i$ for $1\leqs i\leqs N$ and $(M_N,g_N)$ is isometric to the standard metric $(\sphere^2,g_{\sphere^2})$. If 
    \begin{align*}
        f_1=\varphi,\quad f_i=\phi \text{ for }2\leqs i\leqs N-1,\quad f_N=\rho,
    \end{align*}
    then
    \begin{align*}
        \Ric(X_0)&=-\left(n_1\frac{\varphi''}{\varphi}+n\frac{\phi''}{\phi}+2\frac{\rho''}{\rho}\right)X_0,\\
        \Ric(X^{(1)})&\geqs\left[-\frac{\varphi''}{\varphi}+(n_1-1)\frac{1-(\varphi')^2}{\varphi^2}-n\frac{\varphi'\phi'}{\varphi \phi}-2\frac{\varphi'\rho'}{\varphi\rho}\right]X^{(1)},\\
        \Ric(X^{(i)})&\geqs\left[-\frac{\phi''}{\phi}+\frac{n_i-1}{\phi^2}-(n-1)\frac{(\phi')^2}{\phi^2}-n_1\frac{\varphi'\phi'}{\varphi \phi}-2\frac{\phi'\rho'}{\phi\rho}\right]X^{(i)}\ (2\leqs i\leqs N-1),\\
        \Ric(X^{(N)})&=\left[-\frac{\rho''}{\rho}+\frac{1-(\rho')^2}{\rho^2}-n_1\frac{\varphi'\rho'}{\varphi\rho}-n\frac{\phi'\rho'}{\phi\rho}\right]X^{(N)},
    \end{align*}
    where $n=\sum_{i=2}^{N-1}n_i$.
\end{lemma}
\begin{proof}
    It follows from Lemma \ref{lmm-nwarped}.
\end{proof}
We denote $\Ric(\hat{g})$ in the lemma above as $\widehat{\Ric}_{\varphi,\phi,\rho}$.
Noticing the similarity to Lemma \ref{lmm-riccicurvature}, this is why we are able to prove a stronger result. We now state several lemmas similar to those in previous sections without proofs, since they are almost same as those presented earlier.
\begin{lemma}[Model I]\label{lmm2-localmodel1}
Let $n_i \geqs 2 $ $(1\leqs i\leqs N)$. For any $0< \epsilon \leqs \frac{1}{100} $, $0<\delta \leqs \delta_0 (n_i,\epsilon) $ and $0<k<k_0(n_i)$, there exist constants $R(n_i,\epsilon,\delta,k)>0$ and positive functions $ \varphi ,\phi ,\rho $ on $(0,\infty )$, such that
\begin{equation*}
    \begin{array}{lll}
       \varphi |_{(0,1 )} = (1-\epsilon)r, & \phi |_{(0,1)} = \delta, &  \rho'|_{(0,1)}=0,\\
       \varphi|_{[R,+\infty )} = kr,  & \phi|_{[R,+\infty )} =kr, &\rho'|_{[R,+\infty)} =0,
    \end{array}
\end{equation*}
and $\widehat{\Ric}_{\varphi,\phi,\rho} \geqs 0$.
\end{lemma}
\begin{lemma}[Model II]\label{lmm2-localmodel2}
Let $n_i \geqs 2 $ $(1\leqs i\leqs N)$. Then for any $0< \epsilon \leqs \frac{1}{100} $, $\lambda>0$, $0<k<k_0(n_i)$, there are constants $R(n_i,k,\epsilon )>0$, $\delta (n_i,k,\epsilon) >0$ and positive functions $ \varphi ,\phi ,\rho $ on $(0,\infty )$, such that
\begin{equation*}
    \begin{array}{lll}
       \varphi |_{(0,1 )} = kr, & \phi |_{(0,1)} = kr, &  \rho|_{(0,1)}=\lambda, \\
       \varphi|_{[R,+\infty )} = (1-\epsilon ) r,  & \phi|_{[R,+\infty )} = \delta, &\rho|_{[R,+\infty)} =\lambda,
    \end{array}
\end{equation*}
and $\widehat{\Ric}_{\varphi,\phi,\rho} \geqs 0$.
\end{lemma}
\begin{lemma}[Block]\label{lmm2-unit}
	Let $n_i \geqs 2 $ $(1\leqs i\leqs N)$. For any $\epsilon>0$ and $L>1$, then there exists $k=k(n_i)$, $R=R(n_i,\epsilon,L)>1$, $0<\delta<c(n_i,\epsilon)L\iv$, and positive smooth functions $\varphi,\phi,\rho$ on $r\in (0,\infty)$ such that
	\begin{equation*}
		\begin{array}{lll}
		\varphi|_{(0,(LR)\iv)}=kr, & \phi|_{(0,(LR)\iv)}=kr, & \rho|_{(0,(LR)\iv)}=\lambda_1,\\
		\varphi|_{[L\iv,1]}=(1-\epsilon)r, & \phi|_{[L\iv,1]}=\delta, & \rho|_{[L\iv,1]}=\lambda_1,\\
		\varphi|_{[R,\infty)}=kr, & \phi|_{[R,\infty)}=kr,& \rho|_{[R,\infty)}=\lambda_2,
	\end{array}
	\end{equation*}
	 and $\widehat{\Ric}_{\varphi,\phi,\rho}\geqs0$.
\end{lemma}
\begin{proposition}[Connecting]\label{prop2-connecting}
	Given $n_i\geqs2$, $\epsilon>0$, $L>1$ and $1\leqs i_0\leqs N-2$. Consider the warped product with functions
    \begin{align*}
        f_{i_0}=\varphi,\quad f_{i}=\phi \text{ for }1\leqs i\leqs N-2 \text{ and }i\neq i_0,\quad f_N=\rho.
    \end{align*}
  Then there exists $k=k(n_i)$, $R=R(n_i,\epsilon,L)>1$ and positive smooth functions $\varphi,\phi,\rho,f_{N-1}$ on $r\in (0,\infty)$ such that
	\begin{equation*}
		\begin{array}{lllll}
		\text{For } r\in(0,(2L^2R^3)\iv): & \varphi=kr, & \phi=kr, & f_{N-1}=kr,&
  \rho=\lambda_1,\\
            \text{For } r\in[(2L^2R^2)\iv,(2LR^2)\iv]:&
		\varphi=\delta_1, & \phi=\delta_1,&f_{N-1}=(1-\epsilon)r, & \rho=\lambda_1,\\
        \text{For } r\in[(2LR)\iv,(LR)\iv]:&
		\varphi=kr, & \phi=kr, &f_{N-1}=kr, & \rho=\lambda_2,\\
         \text{For } r\in[L\iv,1]:&
		\varphi=(1-\epsilon)r, & \phi=\delta_2, &f_{N-1}=\delta_2, & \rho=\lambda_2\\
        \text{For } r\in[R,\infty):&
		\varphi=kr, & \phi=kr,&f_{N-1}=kr, & \rho=\lambda_3,
	\end{array}
	\end{equation*}
	 where $0<\delta_1<c(n_i,\epsilon)(L^2R^2)\iv$, $0<\delta_2<c(n_i,\epsilon)L\iv$ and $\widehat{\Ric}\geqs0$.

\end{proposition}

Now we are ready to prove Theorem \ref{thm-riemcone}.
\begin{pfthm2} 
Let $\tilde{N}=N+2$ and $(M_{\tilde{N}-1},g_{\tilde{N}-1})\simeq(M_{\tilde{N}},g_{\tilde{N}})\simeq(\sphere^2,g_{\sphere^2})$. The argument is almost same as proof of Theorem \ref{thm-weakregular}. We similarly use Proposition \ref{prop2-connecting} repeatedly to construct $\tilde{g}_i$ by induction, with the only difference being that each time we apply Proposition \ref{prop2-connecting}, we select a different $i_0$ in sequence. Since $(M_{\tilde{N}-1},g_{\tilde{N}-1})\simeq(\sphere^2,g_{\sphere^2})$ and $f_{\tilde{N}-1}=(1-\epsilon_i)r$ on a sequence of intervals, we can still use Lemma \ref{lmm-smoothing} to get $\hat{g_i}$ by smoothing $\tilde{g}_i$. By precompactness, there exists a metric space $(X,d)$ such that
\begin{align*}
	(\hat{M},\hat{g}_i,O_i)\stackrel{GH}{\longrightarrow} (X,d ,x ).
	\end{align*}
 For each $i_0=1,\cdots,N$, we can find $A_i\to\infty$, such that the rescaled metrics $(\hat{M},A_i^2\hat{g}_j,O_j)$ for $j\geqs i$ become
 \begin{align*}
     (f_{i_0})_j|_{(L_i^{-1/2},\frac12L_i^{1/2})}=(1-\epsilon_i)r,\quad (f_k)_j|_{(L_i^{-1/2},\frac12L_i^{1/2})}\leqs \Psi(i\iv) \text{ for }k\neq i_0.
 \end{align*}
 Let $j\to\infty$, we have
\begin{align*}	d_{GH}\left(A_{L_i^{-1/2},\frac12L_i^{1/2}}(X,A_id,x),A_{L_i^{-1/2},\frac12L_i^{1/2}}(C(M_{i_0},g_{i_0}),v_C)\right)\leqs \Psi(i\iv).
	\end{align*}
	Let $i\to\infty$, we have
	\begin{align*}
		(X,A_id,x)\stackrel{GH}{\longrightarrow} (C(M_{i_0},g_{i_0}),v_C).
	\end{align*}
\end{pfthm2}

\bibliographystyle{ytamsalpha}
\bibliography{ref.bib}	

\end{document}